\documentclass{amsproc}
\usepackage{breqn,float,amssymb,pdflscape,rotating}
\makeatletter
\@namedef{subjclassname@2020}{%
  \textup{2020} Mathematics Subject Classification}
\makeatother
\newtheorem{theorem}{Theorem}[section]

\theoremstyle{definition}

\newtheorem{example}[theorem]{Example}

\theoremstyle{remark}

\numberwithin{equation}{section}
\allowdisplaybreaks
\begin{document}

\title{Finite Sums and Products involving Special Functions}

%    Remove any unused author tags.

%    author one information
\author{Robert Reynolds}
\address[Robert Reynolds]{Department of Mathematics and Statistics, York University, Toronto, ON, Canada, M3J1P3}
\email[Corresponding author]{milver73@gmail.com}
\thanks{}

%    author two information
%\author{ Allan Stauffer}
%\address[Allan Stauffer]{Department of Mathematics and Statistics, York University, Toronto, ON, Canada, M3J1P3}
%\email{stauffer@yorku.ca}
%\thanks{This research is supported by NSERC Canada under Grant 504070}

\subjclass[2020]{Primary  30E20, 33-01, 33-03, 33-04}

\keywords{Finite sum, finite product, Special function, contour integral, Nielsen, Gauss}

\date{}

\dedicatory{}

\begin{abstract}
Various product and sum relationships are established using special functions, specifically involving Special functions. These relationships are derived from formulas inspired by the finite sum that incorporates the Hurwitz-Lerch zeta function.
\end{abstract}

\maketitle
\section{Introduction}
Finite sums and products involving Special functions refers to mathematical expressions that involve the summation or multiplication of a finite number of terms, where the terms themselves are represented by special functions. These special functions are typically mathematical functions that have specific properties or applications in various branches of mathematics and physics. The finite sums and products involving special functions can arise in the study of series, sequences, combinatorics, and other areas where these specialized mathematical functions are utilized.\\\\
The finite sum of the Hurwitz-Lerch zeta function finds applications in various areas of mathematics and physics. A few examples are; Number theory section (27.17) in \cite{dlmf}, where the finite sum of the Hurwitz-Lerch zeta function is closely related to the theory of special values of $L$-functions and Dirichlet series. It has been used in the study of modular forms, elliptic curves, and other topics in number theory. Another example is in Analytic number theory \cite{Rademacher}, where the finite sum of the Hurwitz-Lerch zeta function appears in the analysis of certain arithmetic functions, such as the divisor function and the Riemann zeta function. It can be used to study the distribution of prime numbers and investigate properties of arithmetic sequences. In Statistical physics, see section (24.18) in \cite{dlmf}, the finite sum of the Hurwitz-Lerch zeta function arises in the study of thermodynamic properties of physical systems, particularly in the context of lattice models and quantum field theory. It is used to analyze partition functions, correlation functions, and critical phenomena. In Quantum mechanics, see section (36.14) in \cite{dlmf}, the finite sum of the Hurwitz-Lerch zeta function is encountered in the calculation of energy levels and wave functions of quantum mechanical systems with certain potentials or boundary conditions. It provides a mathematical tool for solving differential equations in quantum mechanics. In the area of Special functions and mathematical physics, \cite{bell,dlmf,atlas,andrews}, the finite sum of the Hurwitz-Lerch zeta function is part of a broader class of special functions used in mathematical physics. It appears in the context of integral transforms, differential equations, and solutions to boundary value problems. These applications demonstrate the versatility and utility of the finite sum of the Lerch function in various branches of mathematics and physics.\\\\
The book authored by Prudnikov et al. \cite{prud1} provides a comprehensive compilation of both indefinite and definite integrals, as well as finite and infinite sums and products involving elementary and special functions. The book offers an extensive inventory of these mathematical expressions, catering to a wide range of applications and research areas.
In this current study, we aim to build upon previous research that focused on the finite sum of special functions. To achieve this, we employ the contour integral method \cite{reyn4}, specifically applied to equation (4.4.7.12) mentioned in \cite{prud1}. Consequently, we obtain the contour integral representation as a result given by;
\begin{multline}\label{eq:contour}
\frac{1}{2\pi i}\int_{C}\sum_{p=0}^{n}\left(-2^{-p} a^w w^{-k-1} \tan \left(2^{-p-1} (m+w)\right) \sec \left(2^{-p}
   (m+w)\right)\right)dw\\\\
   =\frac{1}{2\pi i}\int_{C}\left(2^{-n} a^w w^{-k-1} \left(\csc \left(2^{-n} (m+w)\right)-2^{n+1}
   \csc (2 (m+w))\right)\right)dw
\end{multline}
where $a,m,k\in\mathbb{C},Re(m+w)>0,n\in\mathbb{Z^{+}}$. Using equation (\ref{eq:contour}) the main Theorem to be derived and evaluated is given by
\begin{multline}
\sum_{p=0}^{n}2^{-p} e^{i m 2^{-p}} \left(\left(i 2^{-p}\right)^k e^{i m 2^{-p}} \Phi
   \left(-e^{i 2^{1-p} m},-k,1-i 2^{p-1} \log (a)\right)\right. \\ \left.
   -\left(i2^{-p-1}\right)^k \Phi \left(-e^{i 2^{-p} m},-k,1-i 2^p \log
   (a)\right)\right)\\\\
=   i \left(i 2^{-n}\right)^{k+1} e^{i m 2^{-n}} \Phi \left(e^{i
   2^{1-n} m},-k,\frac{1}{2} \left(1-i 2^n \log (a)\right)\right)\\
   +i^k 2^{k+1}
   e^{2 i m} \Phi \left(e^{4 i m},-k,\frac{1}{2}-\frac{1}{4} i \log
   (a)\right)
\end{multline}
The expression provided involves the variables $k$, $a$, $m$, which can be any complex numbers, and $n$, which represents any positive integer. This expression is subsequently utilized to obtain specific instances using trigonometric functions. The derivations employed in this process are based on the approach presented in our previous work referenced as \cite{reyn4}. This approach incorporates a variant of the generalized Cauchy's integral formula, as given by;
\begin{equation}\label{intro:cauchy}
\frac{y^k}{\Gamma(k+1)}=\frac{1}{2\pi i}\int_{C}\frac{e^{wy}}{w^{k+1}}dw,
\end{equation}
In this context, we consider variables $y$ and $w$ belonging to the set of complex numbers ($\mathbb{C}$), while $C$ represents a general open contour in the complex plane. The bilinear concomitant, as described in the reference \cite{reyn4}, vanishes at the endpoints of the contour. This approach involves utilizing a specific form of equation (\ref{intro:cauchy}), multiplying both sides by a function, and subsequently summing the resulting finite terms on both sides. As a result, a finite sum expressed in terms of a contour integral is obtained. Additionally, by multiplying both sides of equation (\ref{intro:cauchy}) by another function and summing infinitely on both sides, the contour integrals in both equations become identical.
\section{The Hurwitz-Lerch Zeta Function}

In our analysis, we employ equation (1.11.3) from the reference \cite{erd}, where $\Phi(z,s,v)$ denotes the Hurwitz-Lerch zeta function. This function serves as a generalization of both the Hurwitz zeta function $\zeta(s,v)$ and the Polylogarithm function $\text{Li}_n(z)$. The Lerch function can be expressed using a series representation, which is given by.

\begin{equation}\label{knuth:lerch}
\Phi(z,s,v)=\sum_{n=0}^{\infty}(v+n)^{-s}z^{n}
\end{equation}
where $|z|<1, v \neq 0,-1,-2,-3,..,$ and is continued analytically by its integral representation given by

\begin{equation}\label{knuth:lerch1}
\Phi(z,s,v)=\frac{1}{\Gamma(s)}\int_{0}^{\infty}\frac{t^{s-1}e^{-vt}}{1-ze^{-t}}dt=\frac{1}{\Gamma(s)}\int_{0}^{\infty}\frac{t^{s-1}e^{-(v-1)t}}{e^{t}-z}dt
\end{equation}
where $Re(v)>0$, and either $|z| \leq 1, z \neq 1, Re(s)>0$, or $z=1, Re(s)>1$.
\section{Contour Integral Representation for the Finite Sum of the Hurwitz-Lerch Zeta Functions}
In this section we derive the contour integral representations of the left-hand side and right-hand side of equation (\ref{eq:contour}) in terms of the Hurwtiz-Lerch zeta and trigonometric functions.
\subsection{Derivation of the left-hand side first contour integral}
We employ the methodology described in reference \cite{reyn4}. By utilizing equation (\ref{intro:cauchy}), our initial step involves substituting $\log (a)+i 2^{-p} (y+1)$ and subsequently multiplying both sides of the equation by $-i 2^{1-p} (-1)^y e^{i m 2^{-p} (y+1)}$. We then proceed to perform finite and infinite summations over the ranges $p\in[0,n]$ and $y\in [0,\infty)$ respectively. Finally, we simplify the resulting expression in terms of the Hurwitz-Lerch Zeta function, yielding
\begin{multline}\label{fsci1}
-\sum_{p=0}^{n}\frac{i 2^{k-p+1} \left(i 2^{-p-1}\right)^k e^{i m 2^{-p}} \Phi
   \left(-e^{i 2^{-p} m},-k,1-i 2^p \log (a)\right)}{\Gamma(k+1)}\\\
   =-\frac{1}{2\pi i}\sum_{y=0}^{\infty}\sum_{p=0}^{n}\int_{C}(-1)^y a^w w^{-k-1} e^{i 2^{-p} (y+1) (m+w)}dw\\\\
   =-\frac{1}{2\pi i}\int_{C}\sum_{p=0}^{n}\sum_{y=0}^{\infty}(-1)^y a^w w^{-k-1} e^{i 2^{-p} (y+1) (m+w)}dw\\\\
   =\frac{1}{2\pi i}\int_{C}\sum_{p=0}^{n}\left(2^{-p} a^w w^{-k-1} \tan
   \left(2^{-p-1} (m+w)\right)-i 2^{-p} a^w w^{-k-1} \right)dw
\end{multline}
from equation (1.232.1) in \cite{grad} where $Re(w+m)>0$ and $Im\left(m+w\right)>0$ in order for the sums to converge. We apply Tonelli's theorem for multiple sums, see page 177 in \cite{gelca} as the summands are of bounded measure over the space $\mathbb{C} \times [0,n] \times [0,\infty) $.
\subsection{Derivation of the additional contour integral}
Using equation (\ref{intro:cauchy}) we replace $y$ with $\log (a)$ and multiply both sides by $-i 2^{-p}$ take the finite sum over $p\in [0,n]$ and simplify to get;
\begin{equation}\label{fsci2}
-\frac{1}{2\pi i}\sum_{p=0}^{n}\frac{i 2^{-p} \log ^k(a)}{\Gamma(k+1)}=-\frac{1}{2\pi i}\int_{C}\sum_{p=0}^{n}i 2^{-p} a^w w^{-k-1}dw
\end{equation}
\subsection{Derivation of the left-hand side second contour integral}
We employ the methodology described in reference \cite{reyn4}. By utilizing equation (\ref{intro:cauchy}), our initial step involves substituting $\log (a)+i 2^{1-p} (y+1)$ and subsequently multiplying both sides of the equation by $-i 2^{1-p} (-1)^y e^{i m 2^{1-p} (y+1)}$. We then proceed to perform finite and infinite summations over the ranges $p\in[0,n]$ and $y\in [0,\infty)$ respectively. Finally, we simplify the resulting expression in terms of the Hurwitz-Lerch Zeta function, yielding
\begin{multline}\label{fsci3}
\sum_{p=0}^{n}\frac{i 2^{k-p+1} \left(i 2^{-p}\right)^k e^{i m 2^{1-p}} \Phi \left(-e^{i
   2^{1-p} m},-k,1-i 2^{p-1} \log (a)\right)}{\Gamma(k+1)}\\\
   =-\frac{1}{2\pi i}\sum_{y=0}^{\infty}\sum_{p=0}^{n}\int_{C}(-1)^y a^w w^{-k-1} e^{i 2^{1-p} (y+1) (m+w)}dw\\\\
   =-\frac{1}{2\pi i}\int_{C}\sum_{p=0}^{n}\sum_{y=0}^{\infty}(-1)^y a^w w^{-k-1} e^{i 2^{1-p} (y+1) (m+w)}dw\\\\
   =\frac{1}{2\pi i}\int_{C}\sum_{p=0}^{n}\left(-2^{-p} a^w w^{-k-1} \tan
   \left(2^{-p} (m+w)\right)+i 2^{-p} a^w w^{-k-1} \right)dw
\end{multline}
from equation (1.232.1) in \cite{grad} where $Re(w+m)>0$ and $Im\left(m+w\right)>0$ in order for the sums to converge. We apply Tonelli's theorem for multiple sums, see page 177 in \cite{gelca} as the summands are of bounded measure over the space $\mathbb{C} \times [0,n] \times [0,\infty) $.
\subsection{Derivation of the additional contour integral}
Using equation (\ref{intro:cauchy}) we replace $y$ with $\log (a)$ and multiply both sides by $i 2^{-p}$ take the finite sum over $p\in [0,n]$ and simplify to get;
\begin{equation}\label{fsci4}
\frac{1}{2\pi i}\sum_{p=0}^{n}\frac{i 2^{-p} \log ^k(a)}{\Gamma(k+1)}=\frac{1}{2\pi i}\int_{C}\sum_{p=0}^{n}i 2^{-p} a^w w^{-k-1}dw
\end{equation}
\subsection{Derivation of the right-hand side first contour integral}
We use the method in \cite{reyn4}. Using equation (\ref{intro:cauchy})  we first replace $\log (a)+2 i (2 y+1)$ and multiply both sides by $4 i e^{2 i m (2 y+1)}$ then take the finite and infinite sum over $y\in [0,\infty)$ and simplify in terms of the Hurwitz-Lerch Zeta function to get
\begin{multline}\label{fsci5}
\frac{(4 i)^{k+1} e^{2 i m} \Phi \left(e^{4 i
   m},-k,\frac{1}{2}-\frac{1}{4} i \log (a)\right)}{\Gamma(k+1)}\\\
   =\frac{1}{2\pi i}\sum_{y=0}^{\infty}\int_{C}4 i a^w e^{2 i (2 y+1) (m+w)}dw\\\\
   =-\frac{1}{2\pi i}\int_{C}\sum_{y=0}^{\infty}4 i a^w e^{2 i (2 y+1) (m+w)}dw\\\\
   =-\frac{1}{2\pi i}\int_{C}2 a^w w^{-k-1} \csc (2(m+w))dw
\end{multline}
from equation (1.232.3) in \cite{grad} where $Re(w+m)>0$ and $Im\left(m+w\right)>0$ in order for the sums to converge. We apply Tonelli's theorem for multiple sums, see page 177 in \cite{gelca} as the summands are of bounded measure over the space $\mathbb{C}  \times [0,\infty) $.
\subsection{Derivation of the right-hand side second contour integral}
We use the method in \cite{reyn4}. Using equation (\ref{intro:cauchy})  we first replace $\log (a)+i 2^{-n} (2 y+1)$ and multiply both sides by $-i 2^{1-n} e^{i m 2^{-n} (2 y+1)}$ then take the finite and infinite sum over $y\in [0,\infty)$ and simplify in terms of the Hurwitz-Lerch Zeta function to get
\begin{multline}\label{fsci6}
-\frac{i 2^{k-n+1} \left(i 2^{-n}\right)^k e^{i m 2^{-n}} \Phi \left(e^{i
   2^{1-n} m},-k,2^{n-1} \left(2^{-n}-i \log (a)\right)\right)}{\Gamma(k+1)}\\\
   =-\frac{1}{2\pi i}\sum_{y=0}^{\infty}\int_{C}i 2^{1-n} a^w e^{i 2^{-n} (2 y+1) (m+w)}dw\\\\
   =-\frac{1}{2\pi i}\int_{C}\sum_{y=0}^{\infty}i 2^{1-n} a^w e^{i 2^{-n} (2 y+1) (m+w)}dw\\\\
   =-\frac{1}{2\pi i}\int_{C}2^{-n} a^ww^{-k-1} \csc \left(2^{-n} (m+w)\right)dw
\end{multline}
from equation (1.232.3) in \cite{grad} where $Re(w+m)>0$ and $Im\left(m+w\right)>0$ in order for the sums to converge. We apply Tonelli's theorem for multiple sums, see page 177 in \cite{gelca} as the summands are of bounded measure over the space $\mathbb{C}  \times [0,\infty) $.
\section{Hurwitz-Lerch zeta function identity}
In this section we will derive the finite sum of Hurwitz-Lerch Zeta functions in terms of the Hurwitz-Lerch Zeta function along with its composite functions.
\begin{theorem}
For all $k,a,m\in\mathbb{C}, n\in\mathbb{Z^{+}}$ then,
\begin{multline}\label{fslf}
\sum_{p=0}^{n}2^{-p} e^{i m 2^{-p}} \left(\left(i 2^{-p}\right)^k e^{i m 2^{-p}} \Phi
   \left(-e^{i 2^{1-p} m},-k,1-i 2^{p-1} \log (a)\right)\right. \\ \left. 
   -\left(i2^{-p-1}\right)^k \Phi \left(-e^{i 2^{-p} m},-k,1-i 2^p \log
   (a)\right)\right)\\
   =i \left(i 2^{-n}\right)^{k+1} e^{i m 2^{-n}} \Phi \left(e^{i
   2^{1-n} m},-k,\frac{1}{2} \left(1-i 2^n \log (a)\right)\right)\\
   +i^k 2^{k+1}
   e^{2 i m} \Phi \left(e^{4 i m},-k,\frac{1}{2}-\frac{1}{4} i \log
   (a)\right)
\end{multline}
\end{theorem}
\begin{proof}
With respect to equation (\ref{eq:contour}) and observing the addition of the right-hand sides of relations (\ref{fsci1}), (\ref{fsci2}), (\ref{fsci3}) and (\ref{fsci4}), and the addition of relations (\ref{fsci5}) and (\ref{fsci6}) are identical; hence, the left-hand sides of the same are identical too. Simplifying with the Gamma function yields the desired conclusion.
\end{proof}
\begin{example}
The degenerate case.
\begin{equation}
\sum_{p=0}^{n}2^{-p-1} \tan \left(m 2^{-p-1}\right) \sec \left(m 2^{-p}\right)=\csc (2
   m)-2^{-n-1} \csc \left(m 2^{-n}\right)
\end{equation}
\end{example}
\begin{proof}
Use equation (\ref{fslf}) and set $k=0$ and simplify using entry (2) in Table below (64:12:7) in \cite{atlas}. Note equation (4.4.7.12) in \cite{prud1} is in error.
\end{proof}
\begin{example}
The product of the ratio of cosine functions in terms of the ratio of tangent functions.
\begin{equation}
\prod_{p=0}^{n}\frac{\cos \left(2^{-p} m\right) }{\cos \left(2^{-p} r\right)}\left(\frac{\cos \left(2^{-1-p} r\right)}{\cos \left(2^{-1-p}
   m\right)}\right)^2=\frac{\tan \left(2^{-1-n} m\right) \tan (r)}{\tan (m) \tan
   \left(2^{-1-n} r\right)}
\end{equation}
\end{example}
\begin{proof}
Use equation (\ref{fslf}) and form a second equation by replacing $m\to r$ take the difference of both these equations then set $k=-1,a=1$ and simplify using entry (3) of Section (64:12) in \cite{atlas}.
\end{proof}
\begin{example}
A functional equation for the Hurwitz-Lerch zeta function.
\begin{multline}
\Phi (z,s,a)=8^{-s} \left(4^s z \Phi \left(-z^2,s,\frac{a+1}{2}\right)+4^s \Phi
   \left(z^2,s,\frac{a}{2}\right)\right. \\ \left.
   -2 z^3 \left(2^s \Phi \left(-z^4,s,\frac{a+3}{4}\right)-2 \Phi
   \left(z^8,s,\frac{a+3}{8}\right)\right)\right)
\end{multline}
\end{example}
\begin{proof}
Use equation (\ref{fslf}) and set $n=1,m=2\log(-z)/i,k=-s,a=e^{(a-1)i/2}$ and simplify..
\end{proof}
\begin{example}
The product of the ratio of cosine functions in terms of the ratio of tangent functions. 
\begin{equation}
\prod_{p=0}^{n}\frac{\cos ^3\left(2^{-1-p} x\right)}{\cos ^2\left(2^{-2-p} x\right) \cos \left(2^{-p} x\right)}\\
=\frac{\tan
   (x) \tan \left(2^{-2-n} x\right)}{\tan \left(\frac{x}{2}\right) \tan \left(2^{-1-n} x\right)}
\end{equation}
\end{example}
\begin{proof}
Use equation (\ref{fslf}) and set $k=1,a=1,m=x$ and simplify using the method in section (8.1) in \cite{reyn_ejpam}.
\end{proof}
\begin{example}
The product of the exponential function and ratio of cosine functions in terms of the ratio of tangent functions.
\begin{multline}
\prod_{p=0}^{n}\cos ^2\left(2^{-p-2} x\right) \cos \left(2^{-p} x\right) \sec ^3\left(2^{-p-1} x\right)\\
 \exp \left(-2^{1-p}
   \left(\cos \left(2^{-p-1} x\right)+\cos \left(3\ 2^{-p-1} x\right)-3 \cos \left(2^{-p} x\right)+1\right) \csc
   \left(2^{1-p} x\right)\right)\\
   =\tan \left(\frac{x}{2}\right) \cot (x) \tan \left(2^{-n-1} x\right) \cot
   \left(2^{-n-2} x\right)\\
    \exp \left(2^{-n} \left(\csc \left(2^{-n} x\right)-\csc \left(2^{-n-1}
   x\right)\right)+\tan \left(\frac{x}{2}\right)-\tan (x)+\cot \left(\frac{x}{2}\right)-\cot (x)\right)
\end{multline}
\end{example}
\begin{proof}
Use equation (\ref{fslf}) and set $k=1,a=e,m=x$ and simplify using the method in section (8.1) in \cite{reyn_ejpam}.
\end{proof}
\begin{example}
Sum of the log-gamma function.
\begin{multline}
\sum_{p=0}^{n}2^{-p} \left(2 \text{log$\Gamma $}\left(-i 2^{p-2} \log (a)\right)-2
   \text{log$\Gamma $}\left(-i 2^{p-1} \log (a)\right)\right. \\ \left.
   -2 \text{log$\Gamma
   $}\left(\frac{1}{4} \left(-i 2^p \log (a)-2\right)\right)+2 \text{log$\Gamma
   $}\left(\frac{1}{2} \left(-i 2^p \log (a)-1\right)\right)+\log \left(\frac{2
   \left(2^p \log (a)-i\right)^2}{\left(2^p \log (a)-2
   i\right)^2}\right)\right)\\
   =2^{-n-1} \left(-2^n \left(8 \text{log$\Gamma
   $}\left(-\frac{1}{4} i \log (a)-\frac{1}{2}\right)+\log (a) \left(2 i \log
   \left(i 2^{-n}\right)+\pi -2 i \log (2)\right)\right.\right. \\ \left.\left.
   +8 \log (-2-i \log (a))-4 \log
   (32 \pi )\right)+4 \text{log$\Gamma $}\left(\frac{1}{2} \left(-i 2^n \log
   (a)-1\right)\right)\right. \\ \left.
   +4 \log \left(-1-i 2^n \log (a)\right)-2 \log (\pi )-6 \log
   (2)\right)
\end{multline}
\end{example}
\begin{proof}
Use equation (\ref{fslf}) and set $m=0$, then simplify using equation (25.14.2) in \cite{dlmf}. Next take the first partial derivative with respect to $k$ and set $k=0$ and simplify using equation (25.11.18) in \cite{dlmf}.
\end{proof}
\begin{example}
Sum of the log-gamma function alternate form.
\begin{multline}\label{eq:loggamma}
\sum_{p=0}^{n}2^{-p} \left(2 \text{log$\Gamma $}\left(2^{p-2} a\right)-2 \text{log$\Gamma $}\left(2^{p-1} a\right)-2
   \text{log$\Gamma $}\left(\frac{1}{4} \left(2^p a-2\right)\right)\right. \\ \left.
   +2 \text{log$\Gamma $}\left(\frac{1}{2}
   \left(2^p a-1\right)\right)+\log \left(\frac{2 \left(a 2^p-1\right)^2}{\left(a
   2^p-2\right)^2}\right)\right)\\
   =2^{-n} \left(2 \text{log$\Gamma $}\left(\frac{1}{2} \left(2^n a-1\right)\right)+2
   \log \left(\frac{a 2^n-1}{2 \sqrt{2 \pi }}\right)\right)-4 \text{log$\Gamma $}\left(\frac{a-2}{4}\right)\\
   +a \log
   \left(2^{-n-1}\right)+4 \log \left(\frac{4 \sqrt{2 \pi }}{a-2}\right)
\end{multline}
\end{example}
\begin{proof}
Use equation (\ref{fslf}) and set $m=0,a=e^{ai}$ then simplify using equation (25.14.2) in \cite{dlmf} Next take the first partial derivative with respect to $k$ and set $k=0$ and simplify using equation (25.11.18) in \cite{dlmf}.
\end{proof}
\begin{example}
Sum involving the  digamma function.
\begin{multline}\label{eq:loggamma1}
\sum_{p=0}^{n}\left(\frac{4}{a 2^p \left(a 2^p-3\right)+2}-\psi ^{(0)}\left(2^{p-2} a\right)+2
   \psi ^{(0)}\left(2^{p-1} a\right)+\psi ^{(0)}\left(\frac{1}{4} \left(2^p
   a-2\right)\right)\right. \\ \left.
   -2 \psi ^{(0)}\left(\frac{1}{2} \left(2^p
   a-1\right)\right)\right)\\
   =-2 \left(\frac{2}{a 2^n-1}+\psi ^{(0)}\left(\frac{1}{2}
   \left(2^n a-1\right)\right)-\frac{4}{a-2}-\psi
   ^{(0)}\left(\frac{a-2}{4}\right)+\log \left(2^{-n-1}\right)\right)
\end{multline}
\end{example}
\begin{proof}
Use equation (\ref{eq:loggamma}) and take the first partial derivative with respect to $a$ and simplify using equation (5.2.2) in \cite{dlmf}. A similar form is given in equation (2.1) in \cite{nielsen}.
\end{proof}
\begin{example}
Log-gamma transformation.
\begin{multline}
\log \left(\Gamma \left(\frac{a}{4}\right)\right)+\log \left(\Gamma
   \left(\frac{a-2}{4}\right) \sqrt{\frac{\Gamma
   \left(\frac{a-1}{2}\right)}{\Gamma \left(\frac{a}{2}\right) \Gamma
   (a)}}\right)=2 \log \left(\frac{\pi ^{3/8} 2^{2-\frac{a}{2}}
   \sqrt[4]{a+\frac{1}{a-1}-3}}{a-2}\right)
\end{multline}
\end{example}
\begin{proof}
Use equation (\ref{eq:loggamma1}) and set $n=1$ and simplify. Similar form is given by equation (8.1) in \cite{nielsen}.
\end{proof}
\begin{example}
Extended Nielsen product form.
\begin{multline}\label{eq:nielsen}
\prod_{p=1}^{n}\left(\frac{2^{-2^{p-1} x} \Gamma \left(\frac{1}{2} \left(2^p
   x+1\right)\right)}{\Gamma \left(\frac{1}{4} \left(2^p
   x+2\right)\right)^2}\right)^{2^{-p}}=\frac{2^{-\frac{n x}{2}-2^{-n}} \left(2^n
   x-1\right)^{2^{-n}} e^{2^{-n} \text{log$\Gamma $}\left(\frac{1}{2} \left(2^n
   x-1\right)\right)}}{\Gamma \left(\frac{x+1}{2}\right)}
\end{multline}
\end{example}
\begin{proof}
Use equation (\ref{eq:loggamma}) and take the exponential of both sides and simplify. Note the exponential of the $\log\Gamma(x)$ on the right-hand side can be simplified for real $x$. Similar forms are given by equations (6.9) and (38.12) in \cite{nielsen}.
\end{proof}
\begin{example}
The infinite limiting case.
\begin{equation}
\prod_{p=1}^{\infty}\left(\frac{2^{-2^{p-1} x} \Gamma \left(\frac{1}{2} \left(2^p
   x+1\right)\right)}{\Gamma \left(\frac{1}{4} \left(2^p
   x+2\right)\right)^2}\right)^{2^{-p}}=\frac{(2 e)^{-x/2} x^{x/2}}{\Gamma
   \left(\frac{x+1}{2}\right)}
\end{equation}
\end{example}
\begin{proof}
Use equation (\ref{eq:nielsen}) and take the limit of the right-hand side as $n\to \infty$ and simplify. This expression holds for $0< Re(x)<1,-1< Im(x)<1$ where $x$ is small. Similar forms are in equation (90.1.10) in \cite{hansen} and equation (8.323) in \cite{grad}.
\end{proof}
\begin{example}
Finite sum involving the generalized Stieltjes constant.
\begin{multline}
\sum_{p=0}^{n}\left(\log \left(i 2^{1-p}\right) \left(H_{2^{p-2} a}-H_{\frac{1}{4} \left(2^p a-2\right)}\right)+2 \log \left(i
   2^{-p}\right) \left(H_{\frac{1}{2} \left(2^p a-1\right)}-H_{2^{p-1} a}\right)\right. \\ \left.
   -\gamma _1\left(2^{p-2}
   a+1\right)+2 \gamma _1\left(2^{p-1} a+1\right)-2 \gamma _1\left(\frac{1}{2} \left(2^p a+1\right)\right)+\gamma
   _1\left(\frac{1}{4} \left(2^p a+2\right)\right)\right)\\
   =\frac{1}{4} \left(-8 \gamma _1\left(\frac{1}{2} \left(2^n
   a+1\right)\right)+8 \log \left(i 2^{-n}\right) \psi ^{(0)}\left(\frac{1}{2} \left(2^n a+1\right)\right)\right. \\ \left.
   +8 \gamma
   _1\left(\frac{a+2}{4}\right)+(-8 \log (2)-4 i \pi ) \psi ^{(0)}\left(\frac{a+2}{4}\right)+4 \log ^2\left(i
   2^{-n}\right)+(\pi -2 i \log (2))^2\right)
\end{multline}
\end{example}
\begin{proof}
Use equation (\ref{fslf}) and set $m=0$ then simplify using equation (25.14.2) in \cite{dlmf}. Next take the first partial derivative with respect to $k$ and take the limit as $k\to-1$ and simplify using equation (25.6.12) in \cite{dlmf}. Similar is given equation (1.1)
in \cite{coffey}.
\end{proof}
\begin{example}
Finite sum involving the polylogarithm function.
\begin{multline}
\sum_{p=0}^{n}2^{-p} \left(\left(2^{-p-1}\right)^k \text{Li}_{-k}\left(-e^{i 2^{-p}
   m}\right)-\left(2^{-p}\right)^k \text{Li}_{-k}\left(-e^{i 2^{1-p}
   m}\right)\right)\\
   =2^{k+1} e^{2 i m} \Phi \left(e^{4 i
   m},-k,\frac{1}{2}\right)-\left(2^{-n}\right)^{k+1} e^{i m 2^{-n}} \Phi
   \left(e^{i 2^{1-n} m},-k,\frac{1}{2}\right)
\end{multline}
\end{example}
\begin{proof}
Use equation (\ref{fslf}) set $a=1$ and simplify using equation (25.14.3 ) in \cite{dlmf}. Similar form is given by equation (3.3.47) in \cite{choi}.
\end{proof}
\begin{example}
Finite product of the exponential of trigonometric functions.
\begin{multline}
\prod_{p=0}^{n}\exp \left(4^{1-p} \left(\sec ^2\left(2^{-p-2} x\right)-3 \sec
   ^2\left(2^{-p-1} x\right)+2 \sec ^2\left(2^{-p} x\right)\right)\right)\\
=\exp
   \left(2^{1-2 n} \left(-\csc ^2\left(2^{-n-2} x\right)+\csc ^2\left(2^{-n-1}
   x\right)+\sec ^2\left(2^{-n-2} x\right)\right.\right. \\ \left.\left.
-\sec ^2\left(2^{-n-1}
   x\right)\right)+32 \cot (x) \csc (x)-32 \cot (2 x) \csc (2 x)\right)
\end{multline}
\end{example}
\begin{proof}
Use equation (\ref{fslf}) and set $k=2,a=1,m=x$ and simplify using the method in section (8.1) in \cite{reyn_ejpam}.
\end{proof}
\section{Conclusion}
This paper introduces a technique that enables the derivation of a finite sum identity incorporating the Hurwitz-Lerch zeta function, along with intriguing sums and products involving Special functions. The obtained results were numerically validated using Mathematica by Wolfram, for parameters in the integrals encompassing real, imaginary, and complex values.
\end{document}